\documentclass[a4paper,12pt]{article}
\usepackage{comment}
\usepackage{cite}
\usepackage{amsmath}
\usepackage{amssymb}
\usepackage{amsfonts}
\usepackage[T1]{fontenc}
\usepackage[utf8]{inputenc}
\usepackage{graphicx}
\usepackage{fancyhdr}
\usepackage{float}
\usepackage{xcolor}
\usepackage{authblk}
\usepackage{mathrsfs}
\usepackage{empheq}
\usepackage[hyphens]{url}
\usepackage{hyperref} 
\usepackage[]{breakurl}
\usepackage[]{amsthm}

\usepackage{geometry}
\newtheorem{theorem}{Theorem}

\begin{document}

\title{Congruence properties of prime sums and Bernoulli polynomials}

\author[$\dagger$]{Jean-Christophe {\sc Pain}$^{1,2,}$\footnote{jean-christophe.pain@cea.fr}\\
\small
$^1$CEA, DAM, DIF, F-91297 Arpajon, France\\
$^2$Universit\'e Paris-Saclay, CEA, Laboratoire Mati\`ere en Conditions Extr\^emes,\\ 
F-91680 Bruy\`eres-le-Ch\^atel, France
}

\date{}

\maketitle

\begin{abstract}
In this article, we derive a congruence property of particular sum rules involving prime numbers. The resulting expression involves Bernoulli numbers and polynomials, for which we obtain, as a consequence, a general congruence relation as well. 
\end{abstract}

\section{Introduction}

Prime sums are of great interest \cite{Zhao2010}, whether for checking calculations, carrying out primality tests, or for cryptography purposes. Some of them are quite surprising, such as (throughout the paper, $p$ is a prime number):
\begin{equation}
\sum_{i,j=1}^{p-1}\left\lfloor\frac{ij}{p}\right\rfloor=\frac{(p-1)^3-(p-1)^2}{4}=\frac{(p-2)(p-1)^2}{4}, 
\end{equation}
or also \cite{Doster1993}:
\begin{equation}
\sum_{k=1}^{(p-1)(p-2)}\left\lfloor(kp)^{1/3}\right\rfloor=\frac{(3p-5)(p-2)(p-1)}{4}. 
\end{equation}
Recently, we obtained a family of sum rules \cite{Pain2023}, generalizing previous works \cite{Doster1993,Doster1997,Koshy1999}. In section \ref{sec1}, we introduce a variant generalizing our sums, for any value of the upper bound of the summation:
\begin{equation}
    \sum_{k=1}^{p-r}\left\lfloor\frac{k^p}{p}\right\rfloor,
\end{equation}
and study its convergence properties. The corresponding relation involves Bernoulli polynomials $B_p(x)$ and numbers $B_p$:
\begin{equation}
    \sum_{k=1}^{p-r}\left\lfloor\frac{k^p}{p}\right\rfloor\equiv \frac{B_{p+1}(r)-B_{p+1}(0)}{p(p+1)}+\frac{(r-1-p)(p-r)}{2p}
\;\;[p],
\end{equation}
which is the first main result of the present work. The congruences of Bernoulli numbers and polynomials were investigated by many authors, such as Sun \cite{Sun1997}. They receive a particular attention, in particular because they are strongly related to the Wilson and Wolstenholme theorems  \cite{Lucas1877,Wolstenholme1892,McIntosh1995,Sun,Hardy2008}. The latter reads
\begin{equation}
    \binom{2p-1}{p-1}=1\;\;[p^3].
\end{equation}
Very recently, Levaillant presented a new proof of Glaisher's formula \cite{Glaisher1900} ($B_p$ is a Bernoulli number) concerning Wilson's theorem modulo $p^2$ \cite{Levaillant2024}:
\begin{equation}
    (p-1)!\equiv p~B_{p-1}-p\;\;[p^2].    
\end{equation}
Her proof relies on $p$-adic numbers and Faulhaber's formula for the sums of powers, as well as properties of Faulhaber's coefficients obtained by Gessel and Viennot \cite{Gessel1989}. Levaillant also proposed  a simpler proof than Sun's one regarding a formula for $(p-1)!$ modulo $p^3$:
\begin{equation}
    (p-1)!\equiv -\frac{p~B_{p-1}}{p-1}+\frac{p~B_{2p-2}}{2(p-1)}-\frac{1}{2}\left(\frac{p~B_{p-1}}{p-1}\right)^2\;\;[p^3].    
\end{equation}

In section \ref{sec2}, we present a congruence relation involving  Bernoulli polynomials, which reads
\begin{equation}
\frac{B_{p+1}(p+1)+B_{p+1}(p)-2B_{p+1}}{p(p+1)}\equiv p^{p-1}\;\;[p].
\end{equation}
and constitutes the second main result of the present work.

\section{A family of prime sums}\label{sec1}

\begin{theorem}

Let $p$ be an odd prime number. Then we have
\begin{equation}
    \sum_{k=1}^{p-r}\left\lfloor\frac{k^p}{p}\right\rfloor\equiv \frac{B_{p+1}(r)-B_{p+1}(0)}{p(p+1)}+\frac{(r-1-p)(p-r)}{2p}
\;\;[p].
\end{equation}


\end{theorem}

\begin{proof}

According to the Fermat little theorem, we have
\begin{equation}
    n^p\equiv n\;\;[p]
\end{equation}
and thus $n^p=n+l p$,where $l$ is an integer, which means that $(n^p-n)/p=l$ and $n^p/p=l+n/p$, with $1\leq n\leq p-1$ and thus
\begin{equation}
\left\lfloor\frac{n^p}{p}\right\rfloor=\frac{n^p-n}{p}.  
\end{equation}
The sum becomes
\begin{equation}
\sum_{k=1}^{p-r}\frac{(k^p-k)}{p}=\frac{1}{p}\sum_{k=1}^{p-1}(k^p-k).    
\end{equation}
We have, for any integer $j$:
\begin{equation}
j^p+(p-j)^p=j^p+\sum_{q=0}^{p-j}\binom{p}{q}p^q(-j)^{p-q}=\sum_{q=1}^{p-j}\binom{p}{q}p^q(-j)^{p-q}+j^p(1+(-1)^p),
\end{equation}
where $\binom{p}{q}$ is the usual binomial coefficient. This gives
\begin{equation}
j^p+(p-j)^p=\sum_{q=1}^{p-j}\binom{p}{q}p^q(-j)^{p-q},
\end{equation}
since $p$ is odd. We have
\begin{equation}
    p!=k!(p-k)!\binom{p}{k}
\end{equation}
and thus $p$ divides $k!(p-k)!\binom{p}{k}$. But $p$ does not divide $k!$ (otherwise $p$ divides one of the factors of $k!$ but they are all $<p$). In the same way, $p$ does not divide $(p-k)!$, thus, according to Euclide's lemma, $p$ divides $\binom{p}{k}$. Since $p$ divides also $p^q$ for $q\geq 1$, one has
\begin{equation}
j^p+(p-j)^p\equiv 0\;\; [p^2]
\end{equation}
which gives, for $j$ varying from 1 to $(p-1)/2$:
\begin{align}
&1^p+(p-1)^p\equiv 0\;\; [p^2]\\
&2^p+(p-2)^p\equiv 0\;\; [p^2]\\
&\vdots\\
&\left(\frac{p-1}{2}\right)^p+\left(\frac{p+1}{2}\right)^p\equiv 0\;\;[p^2]\\
\end{align}
and consequently
\begin{equation}
    \sum_{k=1}^{p-r}k^p\equiv -(p-1)^p-(p-2)^p-\cdots-(p-r+1)^p\;\;[p^2]
\end{equation}
or
\begin{equation}
    \sum_{k=1}^{p-r}k^p\equiv -(-1-2^p-\cdots -(r-1)^p)\;\;[p^2]\equiv \frac{B_{p+1}(r)-B_{p+1}(0)}{p+1}\;\;,
\end{equation}
and 
\begin{equation}
    \sum_{k=1}^{p-r}(k^p-k)\equiv \frac{B_{p+1}(r)-B_{p+1}(0)}{p+1}-\frac{(1+p-r)(p-r)}{2}\;\;[p^2].
\end{equation}
We finally have
\begin{equation}
    \sum_{k=1}^{p-r}\left\lfloor\frac{k^p}{p}\right\rfloor\equiv \frac{B_{p+1}(r)-B_{p+1}(0)}{p(p+1)}+\frac{(r-1-p)(p-r)}{2p}\;\;[p],
\end{equation}
which completes the proof.

\end{proof}

As particular cases, we get

\begin{equation}\label{un}
    \sum_{k=1}^{p-1}\left\lfloor\frac{k^p}{p}\right\rfloor\equiv\frac{(p+1)}{2}\;\;[p]
\end{equation}
or also
\begin{equation}
    \sum_{k=1}^{p-2}\left\lfloor\frac{k^p}{p}\right\rfloor\equiv \frac{1}{p}-\frac{(p-1)(p-2)}{2p}\;\;[p]\equiv \frac{(3-p)}{2}\;\;[p].
\end{equation}

\section{Connection with sum rules and congruence of Bernoulli numbers and polynomials}\label{sec2}

Let $p$ be a prime, and let us consider the sum
\begin{equation}\label{def}
\mathscr{S}_q(p)=\sum_{k=1}^{p-1}\left\lfloor\frac{k^{2q+1}}{p}\right\rfloor,
\end{equation}
for $q\in\mathbb{N}$. We obtained, in a previous work \cite{Pain2023}:
\begin{align}\label{final}
\mathscr{S}_q(p)=\frac{(p-1)(p^{2q}-1)}{2}+\frac{1}{2}\sum_{r=1}^{2q}\frac{(-1)^r}{r+1}\binom{2q+1}{r}\sum_{l=0}^{r}\binom{r+1}{l}B_l~p^{2q+1-l}.
\end{align}
From the definition of Bernoulli polynomials:
\begin{equation}
B_n(x)=\sum_{k=0}^n\binom{n}{k}B_k~x^{n-k},
\end{equation}
one gets
\begin{equation}
\mathscr{S}_q(p)=\frac{B_{2q+2}(-p)+B_{2q+2}(p)-2B_{2q+2}}{2p(2q+2)}-\frac{(p^{2q}+p-1)}{2}.
\end{equation}
Using $B_{2q+2}(-p)=B_{2q+2}(p+1)$, one has
\begin{align}
\mathscr{S}_q(p)=\frac{B_{2q+2}(p+1)+B_{2q+2}(p)-2B_{2q+2}}{2p(2q+2)}-\frac{(p^{2q}+p-1)}{2}
\end{align}
and since $B_{2q+2}(1)=B_{2q+2}(0)=B_{2q+2}$, one can write
\begin{equation}
\mathscr{S}_q(p)=\frac{\left(B_{2q+2}(p+1)-B_{2q+2}(1)\right)+\left(B_{2q+2}(p)-B_{2q+2}(0)\right)}{2p(2q+2)}-\frac{(p^{2q}+p-1)}{2}.
\end{equation}
As an example, in the case $q=1$, one recovers the result given in Refs. \cite{Mathe,Doster1993}:

\begin{equation}
\mathscr{S}_1(p)=\sum_{k=1}^{p-1}\left\lfloor\frac{k^3}{p}\right\rfloor=\frac{(p-1)(p-2)(p+1)}{4}
\end{equation}
and in the cases $q=2$ and 3, one obtains

\begin{equation}
\mathscr{S}_2(p)=\sum_{k=1}^{p-1}\left\lfloor\frac{k^5}{p}\right\rfloor=\frac{1}{12}(p-2)(p-1)(p+1)(2p^2-2 p+3),
\end{equation}
and
\begin{equation}
\mathscr{S}_3(p)=\sum_{k=1}^{p-1}\left\lfloor\frac{k^7}{p}\right\rfloor=\frac{1}{24}(p-2)(p-1)(p+1)(3p^4-6p^3+5 p^2-2 p+6).
\end{equation}
Let $p$ be a prime, and let us consider the particular case where $2q+1=p$. Equation (\ref{def}) becomes:
\begin{equation}\label{defn}
\mathscr{T}(p)=\sum_{k=1}^{p-1}\left\lfloor\frac{k^{p}}{p}\right\rfloor.
\end{equation}
Setting $2q+1=p$ in Eq. (\ref{final}) yields
\begin{align}\label{finaln}
\mathscr{T}(p)=\frac{(p-1)(p^{p-1}-1)}{2}+\frac{1}{2}\sum_{r=1}^{p-1}\frac{(-1)^r}{r+1}\binom{p}{r}\sum_{l=0}^{r}\binom{r+1}{l}B_l~p^{p-l},
\end{align}
or equivalently
\begin{equation}
\mathscr{T}(p)=\frac{B_{p+1}(-p)+B_{p+1}(p)-2B_{p+1}}{2p(p+1)}-\frac{(p^{p-1}+p-1)}{2}.
\end{equation}
Using $B_{p+1}(-p)=B_{p+1}(p+1)$, one has
\begin{align}\label{inte}
\mathscr{T}(p)=\frac{B_{p+1}(p+1)+B_{p+1}(p)-2B_{p+1}}{2p(p+1)}-\frac{(p^{p-1}+p-1)}{2}
\end{align}
and since $B_{p+1}(1)=B_{p+1}(0)=B_{p+1}$, one gets
\begin{equation}
\mathscr{T}(p)=\frac{\left(B_{p+1}(p+1)-B_{p+1}(1)\right)+\left(B_{p+1}(p)-B_{p+1}(0)\right)}{2p(p+1)}-\frac{(p^{p-1}+p-1)}{2}.
\end{equation}
\begin{theorem}

We have the following congruence property
\begin{equation}
\frac{B_{p+1}(p+1)+B_{p+1}(p)-2B_{p+1}}{p(p+1)}\equiv p^{p-1}\;\;[p].
\end{equation}

\end{theorem}
\begin{proof}
    
Using Eq. (\ref{un}), which is a particular case of Theorem 1 and expression (\ref{inte}), one finds
\begin{equation}
\frac{B_{p+1}(p+1)+B_{p+1}(p)-2B_{p+1}}{2p(p+1)}-\frac{(p^{p-1}+p-1)}{2}\equiv\frac{(p+1)}{2}\;\;[p],
\end{equation}
yielding to the congruence property for Bernoulli polynomials
\begin{equation}
\frac{B_{p+1}(p+1)+B_{p+1}(p)-2B_{p+1}}{p(p+1)}\equiv p^{p-1}\;\;[p],
\end{equation}
which completes the proof.

\end{proof}

For instance, one gets
\begin{equation}
\mathscr{T}(3)=\sum_{k=1}^{2}\left\lfloor\frac{k^3}{3}\right\rfloor=2
\end{equation}
and 

\begin{equation}
\mathscr{T}(5)=\sum_{k=1}^{4}\left\lfloor\frac{k^5}{5}\right\rfloor=258,
\end{equation}
as well as
\begin{equation}
\mathscr{T}(7)=\sum_{k=1}^{6}\left\lfloor\frac{k^7}{7}\right\rfloor=10(3\times 7^4-6\times 7^3+5\times 7^2-2\times 7+6)=53820.
\end{equation}
For $p$=3, we have
\begin{equation}
\frac{p+1}{2}=2,  
\end{equation}
and clearly $2\equiv 2\;\;[3]$. For $p$=5, we have
\begin{equation}
\frac{p+1}{2}=3,  
\end{equation}
and $258\equiv 3\;\;[5]$. Finally, for $p$=7, we have
\begin{equation}
\frac{p+1}{2}=4,  
\end{equation}
and $53820\equiv 4\;\;[7]$.

\section{Conclusion}

We established a congruence property for a family of sums involving prime numbers, leading to a general congruence relation for Bernoulli numbers and polynomials. We hope that the method presented here will give rise to the derivation of further properties and sum rules.

\end{document}